\documentclass[a4paper,11pt]{article}
\usepackage[utf8]{inputenc}
\usepackage{graphicx}
\usepackage{amsmath}
\usepackage{amsfonts}
\usepackage{amssymb}
\usepackage{amsthm}
\usepackage{MnSymbol}
\usepackage{pgf,tikz}
\usetikzlibrary{arrows}

\theoremstyle{definition}
\newtheorem{Theorem}{Theorem}[section]
\newtheorem{Definition}[Theorem]{Definition}
\newtheorem{Lemma}[Theorem]{Lemma}
\newtheorem{Proposition}[Theorem]{Proposition}
\newtheorem{Corollary}[Theorem]{Corollary}
\theoremstyle{definition}
\newtheorem{remark}{Remark}
\newtheorem{Question}{Question}

\newcommand{\C}{\mathcal{C}}

\newcommand{\Hm}{\mathcal{H}_m} 

\newcommand{\M}{\mathcal{M}}
\newcommand{\Md}{\mathcal{M}^{*}}
\newcommand{\Hy}{\mathcal{H}}
\newcommand{\B}{\mathcal{B}}
\newcommand{\Bd}{\mathcal{B}^{*}}
\newcommand{\cl}[1]{\sigma(#1)}
\newcommand{\clb}[1]{\sigma_\B(#1)}
\newcommand{\Am}{\mathcal{A}_m}
\newcommand{\A}{\mathcal{A}}

\newcommand{\He}{\Hm^{(1)}}
\def\rk{\operatorname{rk}}
\newcommand{\N}{\mathbb{N}}

\newcommand{\bigs}{\operatorname{bigstar}}
\newcommand{\vstar}{\operatorname{star}}
\renewcommand{\setminus}{\backslash}

\begin{document}

\title{Cryptomorphisms for abstract rigidity matroids}
\author{Emanuele Delucchi\thanks{Department of mathematics, University
  of Fribourg, CH-1700 Fribourg.}
\and
Tim Lindemann\thanks{Department of mathematics,
University of Bremen,
D-28359 Bremen.}
}

\maketitle

\begin{abstract}
  This note contributes to the structure theory of abstract rigidity
  matroids in general dimension.
  In the spirit of classical matroid theory, we prove several cryptomorphic characterizations of abstract
  rigidity matroids (in terms of circuits, cocircuits, bases,
  hyperplanes). Moreover, 
the study of hyperplanes in abstract rigidity matroids leads us to
state (and support with significant evidence) a conjecture about
characterizing the class of abstract rigidity matroids by means of 
certain ``prescribed
  substructures''. We then prove a recursive version of this conjecture.
\footnote{This is an extended version of the second author's bachelor
  thesis at University of Bremen. The first author acknowledges partial support from
  SNSF-Professorship grant PP00P2\_150552/1. The second author acknowledges
  partial support from Swiss European Mobility Programme at University of Fribourg.}
\end{abstract}

\section*{Introduction}

Rigidity matroids are combinatorial structures introduced by Graver in \cite{graver} which, roughly
speaking, model some aspects of the theory of rigid bar-link
framework, whose combinatorial study goes back to work of Laman
\cite{laman}. 

More precisely, a {\em framework} is a (finite) graph $(V,E)$ together with a
straight-line embedding into $\mathbb R^d$, and it is called {\em
  rigid} if the only continuous motions of the embedding of the
vertices which fixes the distance of adjacent vertices are composites
of rotations and translations.
Testing rigidity of a framework involves thus checking for nontrivial
solutions of a system of linear 
equations, every 
equation corresponding to an edge of the graph. Therefore, the rows of this
matrix define a matroid on the ground set $E$. This matroid depends of
course on the chosen embedding, but it always satisfies some abstract
properties (usually defined in terms of the closure operator, see
Definition \ref{DefiRigMat}). An \textit{abstract rigidity matroid}
is any matroid defined on the set of edges of a graph which satisfies
these additional properties. For example, the `usual' cycle
matroid of the given graph is in fact an abstract rigidity matroid,
but abstract rigidity matroids are a much bigger class.  For a comprehensive introduction to the combinatorial
study of rigidity of frameworks we point to the book of Graver,
Servatius and Servatius \cite{combrig}.

Here, we approach the subject
from the point of view of pure matroid theory, which we
briefly introduce in Section 1, and focus on two
main subjects.

First, the fact that classical matroid theory features a web of equivalent
definitions -- so-called {\em cryptomorphic} approaches, see Remark
\ref{rm:crypto} -- has been recognized as one of its main theoretical
strength since at least the seminal work of Crapo and Rota
\cite{CraRo}. In Section 2, we develop the theory of abstract rigidity
matroids in a corresponding way and derive different, equivalent
characterizations, each reflecting one of the classical approaches to
matroid theory.

Second, we focus on structural aspects and ask whether being an
abstract rigidity matroid can be identified as a ``structural''
matroid property, e.g., whether the class of abstract rigidity
matroids can be characterized through some ``prescribed''
substructures. This is a quite standard organizational process in
graph theory and matroid theory: in particular, prominent open
questions and important results about characterizing certain
classes of matroids through ``excluded minors'' abound. In Section 3,
we prove a characterization of abstract rigidity matroids in terms of ``prescribed
classes of hyperplanes'' in all restrictions (Theorem \ref{theo:2dim})
and  conjecture a nonrecursive characterization of abstract rigidity matroids via the
requirement of the existence of a prescribed class of hyperplanes. We
offer some stringent evidence towards this conjecture (e.g.,
Propositions \ref{Hsubset} and \ref{HC1}) as well as in support of the
general significance of the suggested substructures (e.g.\ through
Proposition \ref{connect}).



 \section{Review}\label{sec:rev}
 
In this section we will review some basic
definitions and results about abstract rigidity matroids. In
particular, we will state Viet-Hang Nguyens's combinatorial characterization
of abstract rigidity matroids (Proposition \ref{Prop6}). We will
assume the reader familiar with the basics of matroid theory, and
suggest Oxley's textbook \cite{oxley} for an introduction to the
subject.


\subsection{Graphs}\label{ss:graphs}
We consider a graph to be a pair $(V,E)$ consisting of a finite set
$V$ and a set $E$ of two-element subsets of $V$ (thus, our graphs will
not have loops nor parallel edges). Given two vertices $u,v\in V$ we
will often use $uv$ as a shorthand for $\{u,v\}$. 
For any finite set $W$ we will let $K(W):=\{uv \mid u,v\in W , u\neq v
\}$ so that $(W,K(W))$ is the complete graph on the vertex set
$W$. Given a natural number $n$ we will use the notation $K_n$ to
refer to any (and thus every) $K(W)$ with $\vert W\vert = n$.  
 For any $E\subseteq K(V)$ let $V(E):=\{u\in V \mid uw\in E\text{ for some }w\in V\}$; 
 $V(E)$ is called the \textit{support} of the edge set $E$. 

In the following, we will often consider graphs on a fixed vertex set and
will then, if no confusion can occur, refer to edge-sets as
``graphs''.
 \begin{itemize}
  \item For $v\in V$ let $$\vstar(v):=\{vu \in K(V)\mid u\neq v \} 
\}.$$
 Sets of the form $\vstar(v)$ with $m-1$ arbitrary edges deleted are simply called ``vertex stars minus $m-1$ edges''.
  \item For a set $V'\subseteq V$ such that $|V'|=m$
    let $$\bigs(V'):=K(V)\setminus K(V\setminus V').$$
   Thus, the set $\bigs(V')$ is the edge-set of the graph where every vertex
   $v_0$ not in $V'$ is attached to $K(V')$ by the family of edges
   $\{v_0v' \mid v'\in V'\}$ (one edge to every vertex in $V'$).
 \end{itemize}

 
\begin{figure}[h]
\begin{center}
\begin{tikzpicture}[line cap=round,line join=round,>=triangle 45,x=0.7cm,y=0.7cm]
\clip(-5,-2) rectangle (14,5);
\draw  (-3.5,0)-- (-0.5,0);
\draw  (-3.5,0)-- (-4.5,2);
\draw  (-3.5,0)-- (-3.5,2.5);
\draw  (-3.5,0)-- (-2,3.5);
\draw  (-3.5,0)-- (-0.5,2.5);
\draw  (-3.5,0)-- (0.5,2);
\draw  (0.5,2)-- (-0.5,0);
\draw  (-0.5,0)-- (-0.5,2.5);
\draw  (-2,3.5)-- (-0.5,0);
\draw  (-0.5,0)-- (-3.5,2.5);
\draw  (-4.5,2)-- (-0.5,0);
\draw  (8.5,2)-- (10,3.5)-- (11.5,2)-- (10,0.5);
\draw  (8.5,2)-- (11.5,2)-- (10,0.5)-- (10,3.5);
\draw  (8.5,2)-- (10,0.5)-- (8,0);
\draw  (8,0)-- (8.5,2);
\draw  (8.5,2)-- (8,4);
\draw  (8,4)-- (10,3.5);
\draw  (8,4)-- (11.5,2);
\draw  (8,4)-- (10,0.5);
\draw  (8,0)-- (10,3.5);
\draw  (8,0)-- (11.5,2);
\draw  (8.5,2)-- (12,4);
\draw  (12,4)-- (10,0.5);
\draw  (10,3.5)-- (12,4);
\draw  (12,4)-- (11.5,2);
\draw  (12,0)-- (11.5,2);
\draw  (10,3.5)-- (12,0);
\draw  (12,0)-- (8.5,2);
\draw  (10,0.5)-- (12,0);
\draw  (2,0)-- (6,0);
\draw  (6,0)-- (4,3.46);
\draw  (4,3.46)-- (2,0);
\draw  (3,1.73)-- (4,0);
\draw  (4,0)-- (5,1.73);
\draw  (5,1.73)-- (3,1.73);
\draw  (4,3.46)-- (4,0);
\draw  (2,0)-- (5,1.73);
\draw  (6,0)-- (3,1.73);
\draw (-3.7,-0.22) node[anchor=north west] {$|V| =7 , m =2$};
\draw (2.5,-0.25) node[anchor=north west] {$|V|=6 , m=3$};
\draw (8.5,-0.18) node[anchor=north west] {$|V|=8 , m=4$};
\begin{scriptsize}
\fill [color=black] (-3.5,0) circle (2.0pt);
\fill [color=black] (-0.5,0) circle (2.0pt);
\fill [color=black] (-4.5,2) circle (2.0pt);
\fill [color=black] (-3.5,2.5) circle (2.0pt);
\fill [color=black] (-0.5,2.5) circle (2.0pt);
\fill [color=black] (0.5,2) circle (2.0pt);
\fill [color=black] (-2,3.5) circle (2.0pt);
\fill [color=black] (8.5,2) circle (2.0pt);
\fill [color=black] (10,3.5) circle (2.0pt);
\fill [color=black] (11.5,2) circle (2.0pt);
\fill [color=black] (10,0.5) circle (2.0pt);
\fill [color=black] (12,4) circle (2.0pt);
\fill [color=black] (8,4) circle (2.0pt);
\fill [color=black] (12,0) circle (2.0pt);
\fill [color=black] (8,0) circle (2.0pt);
\fill [color=black] (2,0) circle (2.0pt);
\fill [color=black] (6,0) circle (2.0pt);
\fill [color=black] (4,3.46) circle (2.0pt);
\fill [color=black] (3,1.73) circle (2.0pt);
\fill [color=black] (5,1.73) circle (2.0pt);
\fill [color=black] (4,0) circle (2.0pt);
\end{scriptsize}
\end{tikzpicture}
\caption{Examples of some sets $\bigs(V')$.}
\end{center}
\end{figure}
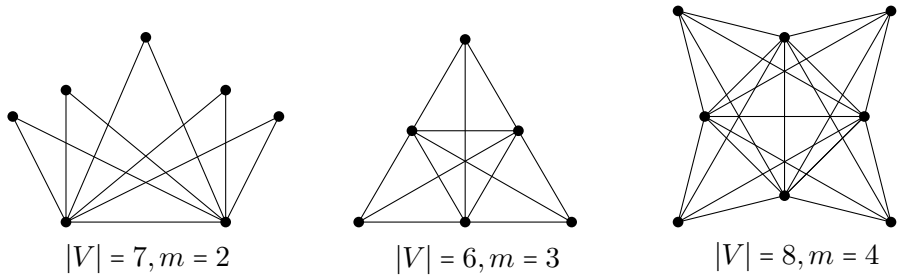


\subsection{Matroids} 

Matroid theory finds its origins in the attempt, by Hassler Whitney,
to define combinatorial structures abstracting some properties of
linear independency in vectorspaces. For instance, it is an easy check
that the set of bases of a vectorspace (say, over a finite field)
satisfies Definition \ref{df:bases} below.

Our goal in this introductory paragraph is to define matroids and some
of the related terminology, and to explain what
matroid theorists mean by {\em cryptomorphism} (see Remark
\ref{rm:crypto}). Indeed the word may sound unusual, but the concept is one that is both useful
in applications and  -- most importantly for us here -- as a theoretical
feature which was singled out as one of the main aspects of interest of
matroid theory ever since at least Crapo and Rota's seminal treaty \cite{CraRo}.


\begin{Definition}\label{df:bases}
Let $S$ be a finite set and let $\B$ be a collection of subsets of $S$ which fulfills
\begin{itemize}
 \item[(i)] $\B \neq \emptyset$.  
 \item[(ii)] If $B_1,B_2\in\B$, then $|B_1|=|B_2|$.
 \item[(iii)] For all $B_1,B_2\in\B$, $x\in B_1\setminus B_2$, there exists $y\in B_2\setminus B_1$, such that \\ 
	      $(B_1\setminus\{x\})\cup \{y\}\in\B$.
\end{itemize}
Then the pair $\M=(S,\B)$ is called a \textit{matroid} and $\B$ is called the collection of \textit{bases} of $\M$.
A set $I\subseteq S$ is called \textit{independent} if there is a
basis $B\in\B$ such that $I\subseteq B$, otherwise
\textit{dependent}. An inclusion-minimal dependent subset of $S$ is
called a {\em circuit} of $\M$. A maximal set which does not contain a
basis is called a \textit{hyperplane}.

The \textit{closure} of a set $A\subseteq S$, denoted by $\clb{A}$, is the intersection of
all hyperplanes containing $A$ (if no such hyperplane exist, the
closure is defined as $A$).
\end{Definition}

A direct consequence of the basis axioms is that the sets of complements of bases
does also fulfill these axioms; the matroid $\Md:=(S,\Bd)$ is called {\em
  dual} to $\M$, and $(\Md)^*=\M$. If a subset of $S$ is a circuit (or a hyperplane, or
a basis etc) of $\Md$ one says that it is a {\em cocircuit} (resp.\
cohyperplane, cobasis) of $\M$. 



We will have use for the following two standard facts, whose proof can be found e.g.\ in \cite{oxley}.

\begin{Lemma} \label{cocircprop}
Let $\M$ be a matroid on a finite set $S$.
\begin{itemize}
 \item[(i)]  If $C$ is a circuit and $C'$ is cocircuit, then $|C\cap C'|\neq 1 $.
 \item[(ii)] $H$ is a hyperplane if and only if $S\setminus H$ is a
   cocircuit.
  \item[(iii)] $F\subseteq S$ is closed if and only if, for every
    circuit $C$, $\vert C\setminus F \vert \leq 1$ implies $C\subseteq
    F$.
\end{itemize}
\end{Lemma}

 
An easy check shows that any function of the form $\clb{\cdot}$
satisfies the following definition.

 \begin{Definition}\label{df:closure}
  Let $S$ be a finite set and let $\sigma: 2^S\to 2^S$ be a function such that for all $A,B\subseteq S$:
  \begin{itemize}
   \item[(i)] $A\subseteq \cl{A}$.
   \item[(ii)] $A\subset B$ implies $\cl{A}\subseteq\cl{B}$.
   \item[(iii)] $\cl{\cl{A}} = \cl{A}$.
   \item[(iv)] If $x,y\in S$ and $x\in\cl{A\cup \{y\}}$, then $y\in\cl{A\cup \{x\}}$.
  \end{itemize}
Then $\cl{\cdot}$ is called a {\em matroid closure operator}, and any
set $A\subseteq S$ is called \textit{closed} (or {\em flat}) if $\cl{A}=A$.
 \end{Definition}

The next theorem is basic and can be found e.g.\ in \cite[Chapter 1]{oxley}

 \begin{Theorem}\label{thm:bc}
The function $\B \mapsto  \clb{\cdot}$ is a bijection between the set
of families $\B$ satisfying
Definition \ref{df:bases} and the set of functions $\cl{\cdot}$
satisfying Definition \ref{df:closure}.
 \end{Theorem}

\begin{remark}[Cryptomorphisms]\label{rm:crypto} In the parlance of matroid theory
  Theorem \ref{thm:bc} is referred to as a {\em cryptomorphism}
  between the definition of matroids in terms of bases and the
  definition in term of closure operator. In fact, just like in
  Definition \ref{df:closure}, one can isolate some distinguishing
  properties of the family of circuits (or of cocircuits, or
  hyperplanes, etc.) of a matroid and prove that any family of sets
  satisfying those formal properties can be obtained as the set of
  circuits (or... etc.) of a matroid defined e.g.\ as in Definition \ref{df:bases}. 
\end{remark}

We close this short presentation of matroids by defining two more
concepts we will have use for later. As a reference we point, again,
to \cite{oxley}.

\begin{Definition}\label{df:restr} Let $\M$ be a matroid on the set $S$ and consider
  $T\subseteq S$. The {\em restriction} of $\M$ to $T$, written
  $\M[T]$, is the matroid with ground set $T$ and closure operator
  defined for each $X\subseteq T$ as $\cl{X}:=\sigma_{\M}(X)\cap T$.
\end{Definition}

\begin{Definition}
  Let $\M=(S,\B)$ be a matroid. The {\em rank} of
  any $X\subseteq S$ is $$\rk(X):=\max\vert \{B\cap X \mid B\in \B
  \}\vert,$$
  i.e., the size of the biggest independent set contained in $X$.
\end{Definition}

\begin{Lemma}
  Let $\M=(S,\sigma)$ be a matroid, and $T\subseteq S$ be a closed
  set of $\M$. Consider any unrefinable chain $T=T_0\subsetneq
  T_1\subsetneq \ldots \subsetneq Tj=S$ of closed sets. Then 
  $$\rk(S)-\rk(T) = j.$$
\end{Lemma}


 
\subsection{Abstract rigidity matroids}

We now are ready to introduce the main character of this
note. As these structures are less classical than graphs of matroids,
we will go into some more detail. Notice that the ground set of an abstract rigidity matroid is the set of edges of
a complete graph and, although perhaps not evident from
our abstract point of view, `rigidity' of a set of edges is meant to
be related to (and
indeed comes from) the concept of rigidity of a bar-and-joints
framework in $m$-space. The reader will perhaps find a useful
intuition in thinking of taking the closure of a certain set of elements
(i.e., edges) of
an abstract rigidity matroid as of increasing the given set by all edges
whose presence would not change the degree of rigidity of the given set.

 \begin{Definition}\label{DefiRigMat}  
 Let $V$ be a finite set and let $\A=(K(V),\sigma)$ be a matroid on $K(V)$ with closure operator $\sigma$.
 A set $E\subseteq K(V)$ is then called \textbf{rigid} (with respect to $\A$) if $\cl{E}=K(V(E))$.
 Let $m\in \mathbb{N}_{>0}$; the matroid $\A$ is called a \textbf{m-dimensional abstract rigidity matroid} if
 \begin{itemize}
  \item[C1.] if $E,F\subseteq K(V)$ and $|V(E)\cap V(F)|<m$,\\ then $\cl{E\cup F}\subseteq K(V(E))\cup K(V(F))$
  \item[C2.] if $E,F\subseteq K(V)$ are rigid and $|V(E)\cap V(F)|\geq m$, then $E\cup F$ is rigid.
 \end{itemize}
\end{Definition}
 Roughly speaking condition (C1) says that edge-sets which do not share enough common vertices cannot unite to a rigid set,
 while (C2) says that rigid sets which are connected through enough vertices form a rigid union. 
 Note that (C1) also states that $\cl{E}\subseteq K(V(E))$ for all $E\subseteq K(V)$.\\
 
 \begin{remark}\label{rem:vsm}
   Recall the notation of Section \ref{ss:graphs} and notice that, if
   $\A$ is an $m$-dimensional rigidity matroid on $K(V)$, every
   $K(V')$ with $\vert V' \vert = m+2$ is a circuit and every vertex
   star minus $m-1$ edges is a cocircuit of $\A$.
 \end{remark}

 \begin{remark}
   For the sake of simplicity, if not otherwise specified, throughout the text we will
   consider $m$-dimensional rigidity matroids on $K(V)$ where
   $m\in\N_{>0}$ and $|V|\geq m+1$ be fixed.   Indeed one easily
   checks that every abstract rigidity matroid on $K(V)$ with
   $|V|\leq m$ is trivial in the sense that every edge-set would be
   an independent set in such a matroid. 
 \end{remark}

 \begin{Definition}
   An edge-set is said to fulfill \textit{Laman's condition} in
   dimension $m$, if for all $F\subseteq E$ with $|V(F)|\geq m$, we
   have $|F|\leq m|V(F)|-{m+1\choose 2}$.
 \end{Definition}

 In fact every independent set in a $m$-dimensional abstract rigidity matroid fulfills Laman's condition.
 \begin{Lemma} \cite[Lemma 2.5.6.]{combrig}\label{lem:rank}
  Let  $U\subseteq V$ and let $\A$ be a $m$-dimensional abstract rigidity matroid on $K(V)$.
  Then \[
  r(K(U))=\left\{
 \begin{array}{cc}
 {|U|\choose 2} 	& \text{if }|U|\leq m+1\\
  m|U|-{m+1\choose 2}	& \text{if }|U|\geq m+1 
 \end{array}  \right.
 \]
\end{Lemma}
 Therefore every $m$-dimensional abstract rigidity matroid on $K(V)$ is of rank $m|V|-{m+1 \choose 2}$,
 and rigid edge-sets on at least $m+1$ vertices have at least $m|V|-{m+1 \choose 2}$ edges. 
 
 \begin{Definition}
   An edge-set which is both rigid and independent is called
   \textit{isostatic}.
 \end{Definition}

 A result in rigidity theory states that $2$-isostatic sets are at
 least $2$-vertex connected \cite[Exercise 4.7.]{combrig}. Our Proposition
 \ref{connect} will prove that in fact every rigid set in a
 $m$-dimensional rigidity matroid is $m$-vertex connected. 
 
Note that every rigid edge-set must have an isostatic subset, so rigidity of a graph can be seen as related to connectivity.  This is emphasized by the following result:
 
\begin{Lemma}[Theorem 3.11.8.\ of \cite{combrig}]
 There is only one $1$-dimensional abstract rigidity matroid on $K(V)$, the cycle matroid on $K(V)$.
 \end{Lemma}

 However, for $m>1$, there can be many different $m$-dimensional
 abstract rigidity matroids on $K(V)$. This makes the theory
 interesting, and is investigated in more detail for example in
 \cite{combrig} and \cite{prism}.
 
Here we will only briefly recall some of the structural aspects of
abstract rigidity matroids in general dimension, especially as related
to the problem of characterizations alternative to Definition \ref{DefiRigMat}
(e.g., towards ``cryptomorphisms'' -- cf. Remark \ref{rm:crypto})

 \begin{Definition} 
 Let $E\subseteq K(V)$ be an edge-set, let $v_1,\ldots v_k \in V(E)$ and  $w\in V\setminus V(E)$.
 Then the set $E\cup\{v_1w,\ldots v_kw\}$ is called a \textbf{$k$-valent-$0$-extension} of $E$.
 \end{Definition}
Intuitively, we are here attaching a vertex to a given edgeset by
means of exacly $k$ new edges.
 
\begin{figure}[h]
\centering
\begin{tikzpicture}[line cap=round,line join=round,>=triangle 45,x=0.8cm,y=0.8cm]
\clip(0,-2) rectangle (9.5,2);
\draw (1,-1)-- (1,1);
\draw (1,-1)-- (2,-1);
\draw (2,-1)-- (3,0);
\draw (1,1)-- (3,0);
\draw [->] (3.96,0.06) -- (5.46,0.06);
\draw (6.35,0.99)-- (8.35,-0.01);
\draw (8.35,-0.01)-- (7.35,-1.01);
\draw (6.35,-1.01)-- (6.35,0.99);
\draw (7.35,-1.01)-- (6.35,-1.01);
\draw (8.35,0.99)-- (6.35,0.99);
\draw (8.35,0.99)-- (8.35,-0.01);

\draw[color=black] (1.27,1.35) node {$v_1$};
\draw[color=black] (3.4,0.24) node {$v_2$};
\draw[color=black] (3.15,1.35) node {$w$};
\draw[color=black] (8.75,0.24) node {$v_2$};
\draw[color=black] (6.61,1.35) node {$v_1$};
\draw[color=black] (8.5,1.35) node {$w$};

\begin{scriptsize}
\fill [color=black] (1,-1) circle (2pt);
\fill [color=black] (1,1) circle (2pt);
\fill [color=black] (3,0) circle (2pt);
\fill [color=black] (3,1) circle (2pt);
\fill [color=black] (2,-1) circle (2pt);
\fill [color=black] (8.35,-0.01) circle (2pt);
\fill [color=black] (7.35,-1.01) circle (2pt);
\fill [color=black] (6.35,-1.01) circle (2pt);
\fill [color=black] (6.35,0.99) circle (2pt);
\fill [color=black] (8.35,0.99) circle (2pt);
\end{scriptsize}
\end{tikzpicture}
\caption{An example of a $2$-valent-$0$-extension}
\end{figure}
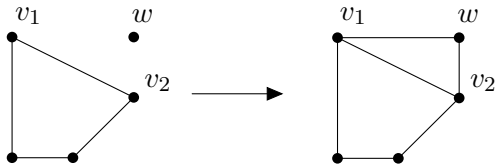

 
 
 \begin{Lemma}[Lemma 0.1 of \cite{km2co}]\label{extensionlemma} 
 Let $\A$ be a matroid on $K(V)$ such that all vertex stars minus $m-1$ edges are cocircuits. Then:
 \begin{itemize}
  \item[(i)]	 no circuit of $\A$ contains a vertex of valence less than $m+1$.
  \item[(ii)]	 for $k\leq m$, every $k$-valent-0-extension of an independent edge-set is independent.
  \item[(iii)]	 for $k\geq m$, every $k$-valent-0-extension of a rigid edge-set is rigid
		  if (C2) holds in $\A$ (with respect to $m$).
 \end{itemize}
\end{Lemma}

\begin{Corollary}
  For any $V'\subseteq V$ with $|V'|=m$, $\bigs(V')$ is a basis in any abstract rigidity matroid of dimension $m$.
\end{Corollary}
\begin{proof}
  For any $V'\subseteq V$ with $|V'|=m$, $\bigs(V')$ is
   the set $K(V')$ with every vertex of $V\setminus V'$ attached
  to it via $m$-valent-$0$-extensions. Note that $K(V')$ is independent and
  rigid, because it is a complete graph on less than $m+2$
  vertices. Hence Lemma \ref{extensionlemma} implies that $\bigs(V')$ is a basis.
\end{proof}


The problem of finding alternative characterizations of abstract
rigidity matroids has been raised in the literature, and we
conclude this short review with the 2010 result of Nguyen which gave
an answer to this problem. This will be the starting point for
our considerations.

 \begin{Proposition}\label{Prop6} \cite[Theorem 2.2.]{nguyen} \cite[Theorem 0.2]{km2co}
  Let $m\in\mathbb{N}$ and let $V$ be a finite set with $|V| \geq m+1$.
  Furthermore, let $\A$ be a matroid on $K(V)$.

  Then, $\A$ is a $m$-dimensional abstract rigidity matroid if and only if any two of the following conditions hold:
  \begin{itemize}
   \item[(i)] all vertex stars minus $m-1$ edges are cocircuits of $\A$.
   \item[(ii)] all $K_{m+2}$ are circuits of $\A$.
   \item[(iii)] $r(K(V))=m|V|-{m+1\choose 2}$.
  \end{itemize}
\end{Proposition}


 \section{Axiomatizations of rigidity matroids}\label{axioms}
 

In this section we will derive several cryptomorphic definitions of
abstract rigidity matroids. Our starting point will be Proposition
\ref{Prop6}, which states conditions on the circuits,
coircuits and on the rank function. We strive for 
characterizations fitting into the classical reformulations
of matroid theory -- in particular: cocircuits, hyperplanes, bases, circuits.

We start by removing from Proposition \ref{Prop6} the reference to rank and circuits.
 
 \begin{Proposition} \ \\
  Let $\A$ be a matroid on $K(V)$.
  Then $\A$ is a $m$-dimensional abstract rigidity matroid if and only if
  \begin{itemize}
  \item[D1.] all cocircuits have strictly more than $|V|-m$ vertices and
  \item[D2.] all vertex stars minus $m-1$ edges are cocircuits.
  
  \end{itemize}

\end{Proposition}
\begin{proof}
	First of all let $\A$ be an $m$-dimensional abstract rigidity matroid.
	We have already seen in Remark \ref{rem:vsm} that all vertex stars minus $m-1$ edges are
        cocircuits in abstract rigidity matroids, so (D2) holds.
	We prove (D1) by contraposition, and so let  $C\subseteq K(V)$ be a set with $|V(C)| \leq |V|-m$.
	If $C$ is empty, then it is independent in any matroid and thus not a cocircuit.
	Let then $C\neq \emptyset$. Thus, $C$ contains an edge with
        endpoints, say, $v_1,v_2\in V(C)$.
	Because of the constraint on the cardinality of $V(C)$, we can
        choose
 $$v_3,\ldots,v_{m+2}\notin V(C).$$
	By Proposition \ref{Prop6}.(ii), $K(\{v_1,\ldots,v_{m+2}\})$ is a circuit
        in $\A$ which intersects the set $C$ in the single edge $v_1v_2$.
	Therefore, $C$ cannot be a cocircuit by Proposition
        \ref{cocircprop}.

        Now let $\A$ be a matroid satisfying (D1) and (D2).
	By Proposition \ref{Prop6}, to prove that $\A$ is an abstract
        rigidity matroid, we only need to verify that $r(\A)=m|V|-{m+1\choose 2}$.
	So choose an $m$-element subset $V'$ of $V$. Since, by (D2), all
        vertex stars minus $m-1$ edges are cocircuits, $\bigs(V')$ is
        independent by Lemma \ref{extensionlemma}. Moreover, its complement is $K(V\setminus V' )$ - a complete graph on $|V|-m$ vertices -
	and thus coindependent. So $\bigs(V')$ is also a spanning set,
        hence a basis, and we can use it to show that the rank of
        $\A$ has the correct value, as follows. 
	  \begin{align*}
	      r(\A) &=r(\bigs(V')) =|K(V')|+m(|V|-m) \\
		&      = {m\choose 2}+m(|V|-m) = m|V|+\left({m\choose 2}-m^2\right) = m|V|-{m+1\choose 2} 
	  \end{align*}
\end{proof}


We now state hyperplane axioms for abstract rigidity matroids, as we
will be concerned extensively with hyperplanes in the next
section. They are easily derived from the cocircuit axioms by complementation.
\begin{Definition}\label{def:delta}
  For $v\in V$ and $A\subseteq V\setminus \{v\}$ let
$$
  \Delta_v^A:= K(\{v\}\cup A) \cup K( V\setminus \{v\}).
  $$
\end{Definition}

\begin{Theorem} \ \\
 Let $\A$ be a matroid on $K(V)$ with hyperplanes $\Hy$.
 Then $\A$ is a $m$-dimensional abstract rigidity matroid if and only if $\Hy$ fulfills:
  \begin{itemize}
  \item[H1.] if $H\in\Hy$ then at most $m-1$ vertices in $V(H)$ have valence $|V|-1$.
  \item[H2.] for all $v\in V$, the sets of the form $\Delta_v^A$
  where $\vert A \vert = m-1$ 
  (i.e., $m-1$-valent $0$-extensions of $K(V\setminus v)$) are hyperplanes.  
 \end{itemize}
\end{Theorem}

We now move to a characterization of abstract rigidity matroids in
terms of bases.

\begin{Theorem} \ \\
  Let $\A$ be a matroid on $K(V)$ with bases $\B$.
  Then $\A$ is a $m$-dimensional abstract rigidity matroid if and only if $\B$ fulfills:
  \begin{itemize}
  \item[B1.] if $B\in\B$ then $V(B)=V$ and every vertex in $B$ has at least valence $m$,
  \item[B2.] for all $V'\subseteq V$ with $|V'|=m$, the set $\bigs(V')$ is a basis.
 \end{itemize}
\end{Theorem}
 \begin{proof}
  Suppose $\A$ is a $m$-dimensional abstract rigidity matroid and
  there is a basis $B$ of $\A$ with a vertex $v$ of degree less than
  $m$. Let $N(v)$ denote the neighbors of $v$ in the graph induced by
  $B$, and choose some $A\in \vstar(v)$ with $\vert A\vert = m-1$ and $N(v)\subseteq A$.
  Then, $B\subseteq \Delta_v^A$, which is a hyperplane by (H2) and
  should then not have maximal rank. Thus no such basis exists and
  (B1) holds. 
  
  To prove (B2) notice that, whenever $\vert V'\vert =m$,
  $\bigs(V')$ is a $m$-valent-0-extensions of some $K_m$, thus
  independent. Moreover, such a $\bigs(V')$ cannot be contained in
  any hyperplane, because $m$ of its vertices have valence $|V|-1$,
  thus it violates (H2).
  
  We now turn to the reverse direction and let $\B$ fulfill (B1) and
  (B2). We will prove that, then, the set of hyperplanes of the given
  matroid satisfies (H1) and (H2). First, by (B2) every set of edges inducing
  a graph with $m$ or more vertices of valence $|V|-1$ contains a
  basis: thus no hyperplane has more than $m-1$ vertices of valence $|V|-1$ and (H1) holds.
  To prove (H2) consider any $v\in V$ and any $A\subseteq V\setminus
  \{v\}$ with $\vert A \vert =m-1$, and let $H:=\Delta_v^A$. Note
  that, by (B1), $H$ cannot contain a basis because $v$ has degree
  $m-1$ in $H$. Moreover, any edge $e$ not in $H$ must have $v$ as an
  endpoint. So, by (B2) $H\cup\{e\}$ contains a basis. We conclude
  that $H$ is a maximal set which does not contain a basis - i.e., a hyperplane - and (H2) holds.	      
 \end{proof}
 We close this section with circuit axioms.

 \begin{Theorem} \ \\
  Let $\A$ be a matroid on $K(V)$ with set of circuits $\C$.
  Then $\A$ is a $m$-dimensional abstract rigidity matroid if and only if $\C$ fulfills:
  \begin{itemize}
  \item[Z1.] if $C\in\C$, no vertex of $V(C)$ has valence less than $m+1$.
  \item[Z2.] all $K_{m+2}$ are circuits.
 \end{itemize}
\end{Theorem}
\begin{proof}
  Let $\A$ be a $m$-dimensional abstract rigidity matroid.
   Then, Proposition \ref{Prop6}.(ii) gives directly (Z2). Proposition
   \ref{Prop6}.(i) says that all vertex stars minus $m-1$ edges are
   cocircuits of $\A$ and so, by Lemma \ref{extensionlemma}.(i), (Z1) holds.
  
  For the reverse implication, let $\C$ satisfy (Z1) and (Z2), and let $\B$ denote the set of bases of $\A$.
  To prove that (B2) holds, choose then any subset $V'\subseteq V$ with $|V'|=m$, and consider the set $\bigs(V')$.
  Every edge $xy$ not in $\bigs(V')$ is contained in the circuit $K(V'\cup \{x,y\})$ and thus 
  the closure of $\bigs(V')$ is $K(V)$. Also, $\bigs(V')$ does not
  contain any circuit: by (Z1),
  all vertices of $V\setminus V'$ have valence $m$, and so any circuit in $\bigs(V')$ must be a subset of $K(V')$, which
  itself is independent because it only contains vertices of valence $m-1$. 
  Summarizing, there is no circuit in $\bigs(V')$ and so it is
  independent, thus a basis and (B2) holds. In order to prove (B1),
  suppose by way of contradiction that there is a basis $B$ with a vertex $v$ of valence less than $m$.
  Then, there is a vertex $w\in V$ such that $vw\notin B$. 
  But $B\cup \{vw\}$ must be dependent: thus it contains a circuit $C$
  which must contain the $vw$.
  In particular, $v$ has at most valence $m$ in $C$, contradicting (Z1). Hence (B1) holds.
\end{proof}


 \section{Hyperplanes of abstract rigidity matroids}
 

All axiomatizations given in Section \ref{axioms} share a similar
structure: they require some ``prescribed substructure'' (e.g.,
conditions (H2), (B2), (Z2), (D2)) and impose some condition on the
valency of vertices (e.g., (H1), (B1), (Z1), (D1)). From a structural
point of view one can ask whether it is possible to characterize
anstract rigidity matroids from a ''prescribed substructure'' point of
view alone.

This question turns out to be intriguing and not trivial. We will 
investigate it from the vantage point of hyperplanes, where the
''prescribed substructures'' are hyperplanes of the form $\Delta_v^A$
for $\vert A \vert = m-1$ (see Definition \ref{def:delta}). We suggest
to look at a bigger family of ``prescribed substructures'', namely the following.

\begin{Definition}\label{df:Hm}
 Given a vertex set $V$ and any $m\geq 1$, define
 {\small \[ \Hm := \{ H = K(V_1) \cup K(V_2)  \mid V_i\nsubseteq V_j \;\mathrm{ for }\; i\neq j,
			  V_1\cup V_2 = V, |V_1\cap V_2|=m-1  \}  \]}

 \end{Definition}

 The family $\Hm$ consists of pairs of sets of edges of complete graphs
 which intersect in a complete graph on $m-1$ vertices. Figure
 \ref{fig:hyex} shows how this sets may look like. 


\begin{figure}[h]

\centering

\begin{tikzpicture}[line cap=round,line join=round,>=triangle 45,x=0.35cm,y=0.35cm]
\clip(-6,-15) rectangle (30,6);
\draw (11.62,3.49)-- (11.62,1.49);
\draw (11.62,1.49)-- (13.52,0.87);
\draw (13.52,0.87)-- (14.7,2.49);
\draw (14.7,2.49)-- (13.52,4.11);
\draw (13.52,4.11)-- (11.62,3.49);
\draw (11.62,1.49)-- (11.62,3.49);
\draw (11.62,3.49)-- (10.05,4.74);
\draw (10.05,4.74)-- (8.1,4.29);
\draw (8.1,4.29)-- (7.24,2.49);
\draw (7.24,2.49)-- (8.1,0.69);
\draw (8.1,0.69)-- (10.05,0.24);
\draw (10.05,0.24)-- (11.62,1.49);
\draw (11.62,1.49)-- (8.1,0.69);
\draw (11.62,1.49)-- (7.24,2.49);
\draw (11.62,1.49)-- (8.1,4.29);
\draw (11.62,1.49)-- (10.05,4.74);
\draw (11.62,3.49)-- (8.1,4.29);
\draw (11.62,3.49)-- (7.24,2.49);
\draw (11.62,3.49)-- (8.1,0.69);
\draw (11.62,3.49)-- (10.05,0.24);
\draw (10.05,4.74)-- (10.05,0.24);
\draw (10.05,4.74)-- (8.1,0.69);
\draw (10.05,4.74)-- (7.24,2.49);
\draw (8.1,4.29)-- (10.05,0.24);
\draw (8.1,4.29)-- (8.1,0.69);
\draw (7.24,2.49)-- (10.05,0.24);
\draw (11.62,1.49)-- (13.52,4.11)-- (13.52,0.87)-- (11.62,3.49)-- (14.7,2.49);
\draw (11.62,1.49)-- (14.7,2.49);
\draw (0.68,-9.56)-- (0.68,-7.56);
\draw (0.68,-7.56)-- (-0.61,-6.02);
\draw (-0.61,-6.02)-- (-2.58,-5.68);
\draw (-2.58,-5.68)-- (-4.31,-6.68);
\draw (-4.31,-6.68)-- (-4.99,-8.56);
\draw (-4.99,-8.56)-- (-4.31,-10.44);
\draw (-4.31,-10.44)-- (-2.58,-11.44);
\draw (-2.58,-11.44)-- (-0.61,-11.09);
\draw (-0.61,-11.09)-- (0.68,-9.56);
\draw (3.68,-8.41)-- (0.68,-7.56);
\draw (3.68,-8.41)-- (0.68,-9.56);
\draw (0.68,-9.56)-- (-4.99,-8.56);
\draw (0.68,-7.56)-- (-4.99,-8.56);
\draw (-4.31,-6.68)-- (-2.58,-11.44);
\draw (0.68,-7.56)-- (-2.58,-5.68);
\draw (0.68,-7.56)-- (-4.31,-6.68);
\draw (0.68,-7.56)-- (-4.31,-10.44);
\draw (0.68,-7.56)-- (-2.58,-11.44);
\draw (0.68,-7.56)-- (-0.61,-11.09);
\draw (0.68,-9.56)-- (-0.61,-6.02);
\draw (0.68,-9.56)-- (-2.58,-5.68);
\draw (0.68,-9.56)-- (-4.31,-6.68);
\draw (0.68,-9.56)-- (-4.31,-10.44);
\draw (0.68,-9.56)-- (-2.58,-11.44);
\draw (-0.61,-6.02)-- (-0.61,-11.09)-- (-4.99,-8.56)-- (-2.58,-5.68)-- (-2.58,-11.44);
\draw (-0.61,-6.02)-- (-4.99,-8.56);
\draw (-0.61,-6.02)-- (-4.31,-6.68);
\draw (-0.61,-6.02)-- (-4.31,-10.44);
\draw (-0.61,-6.02)-- (-2.58,-11.44);
\draw (-0.61,-11.09)-- (-2.58,-5.68);
\draw (-2.58,-5.68)-- (-4.31,-10.44);
\draw (-0.61,-11.09)-- (-4.31,-6.68);
\draw (-0.61,-11.09)-- (-4.31,-10.44);
\draw (-2.58,-11.44)-- (-4.99,-8.56);
\draw (-4.31,-10.44)-- (-4.31,-6.68);
\draw (20.12,0.61)-- (22.12,0.61);
\draw (22.12,0.61)-- (23.12,2.34);
\draw (23.12,2.34)-- (22.12,4.08);
\draw (22.12,4.08)-- (20.12,4.08);
\draw (20.12,4.08)-- (19.12,2.34);
\draw (19.12,2.34)-- (20.12,0.61);
\draw (23.12,2.34)-- (24.12,0.61);
\draw (24.12,0.61)-- (26.12,0.61);
\draw (26.12,0.61)-- (27.12,2.34);
\draw (27.12,2.34)-- (26.12,4.08);
\draw (26.12,4.08)-- (24.12,4.08);
\draw (24.12,4.08)-- (23.12,2.34);
\draw (20.12,4.08)-- (20.12,0.61);
\draw (22.12,4.08)-- (22.12,0.61);
\draw (24.12,4.08)-- (24.12,0.61);
\draw (26.12,4.08)-- (26.12,0.61);
\draw (27.12,2.34)-- (23.12,2.34);
\draw (23.12,2.34)-- (19.12,2.34);
\draw (19.12,2.34)-- (22.12,4.08);
\draw (19.12,2.34)-- (22.12,0.61);
\draw (20.12,0.61)-- (23.12,2.34);
\draw (23.12,2.34)-- (20.12,4.08);
\draw (20.12,4.08)-- (22.12,0.61);
\draw (20.12,0.61)-- (22.12,4.08);
\draw (24.12,4.08)-- (26.12,0.61);
\draw (24.12,0.61)-- (26.12,4.08);
\draw (26.12,0.61)-- (23.12,2.34);
\draw (23.12,2.34)-- (26.12,4.08);
\draw (27.12,2.34)-- (24.12,4.08);
\draw (24.12,0.61)-- (27.12,2.34);
\draw (12.71,-10.83)-- (13.87,-8.37);
\draw (13.87,-8.37)-- (12.67,-5.93);
\draw (12.67,-5.93)-- (10.01,-5.35);
\draw (10.01,-5.35)-- (7.9,-7.06);
\draw (7.9,-7.06)-- (7.92,-9.78);
\draw (7.92,-9.78)-- (10.06,-11.46);
\draw (10.06,-11.46)-- (12.71,-10.83);
\draw (12.71,-10.83)-- (7.9,-7.06);
\draw (7.9,-7.06)-- (13.87,-8.37);
\draw (13.87,-8.37)-- (7.92,-9.78);
\draw (7.92,-9.78)-- (12.67,-5.93);
\draw (12.67,-5.93)-- (10.06,-11.46);
\draw (10.06,-11.46)-- (10.01,-5.35);
\draw (12.71,-10.83)-- (10.01,-5.35);
\draw (13.87,-8.37)-- (10.01,-5.35);
\draw (10.01,-5.35)-- (7.92,-9.78);
\draw (7.92,-9.78)-- (12.71,-10.83);
\draw (12.71,-10.83)-- (12.67,-5.93);
\draw (12.67,-5.93)-- (7.9,-7.06);
\draw (7.9,-7.06)-- (10.06,-11.46);
\draw (10.06,-11.46)-- (13.87,-8.37);
\draw (15.87,-8.37)-- (13.87,-8.37);
\draw (18.91,-8.43)-- (20.91,-10.43);
\draw (20.91,-10.43)-- (22.91,-8.43);
\draw (22.91,-8.43)-- (20.91,-6.43);
\draw (20.91,-6.43)-- (18.91,-8.43);
\draw (18.91,-8.43)-- (22.91,-8.43);
\draw (20.91,-10.43)-- (20.91,-6.43);
\draw (23.91,-8.43)-- (25.91,-6.43);
\draw (25.91,-6.43)-- (25.91,-10.43);
\draw (25.91,-10.43)-- (27.91,-8.43);
\draw (27.91,-8.43)-- (25.91,-6.43);
\draw (25.91,-10.43)-- (23.91,-8.43);
\draw (23.91,-8.43)-- (27.91,-8.43);
\draw (1.17,4.61)-- (-1.83,4.61);
\draw (-1.83,4.61)-- (-0.33,2.01);
\draw (-0.33,2.01)-- (1.17,4.61);
\draw (1.17,4.61)-- (3.67,3.61);
\draw (3.67,3.61)-- (1.67,0.11);
\draw (-1.83,4.61)-- (3.67,2.11);
\draw (3.67,2.11)-- (3.67,3.61);
\draw (-0.33,2.01)-- (1.67,0.11);
\draw (3.67,2.11)-- (1.67,0.11);
\draw (-0.33,2.01)-- (3.67,2.11);
\draw (3.67,2.11)-- (1.17,4.61);
\draw (1.67,0.11)-- (-1.83,4.61);
\draw (-1.83,4.61)-- (3.67,3.61);
\draw (1.67,0.11)-- (1.17,4.61);
\draw (-0.33,2.01)-- (3.67,3.61);
\draw (-4.33,3.61)-- (-1.83,4.61);
\draw (-4.33,3.61)-- (1.17,4.61);
\draw (-4.33,3.61)-- (-0.33,2.01);
\draw (-4.33,3.61)-- (-2.33,0.11);
\draw (-4.33,3.61)-- (-4.33,2.11);
\draw (-2.33,0.11)-- (-4.33,2.11);
\draw (-2.33,0.11)-- (-1.83,4.61);
\draw (-2.33,0.11)-- (1.17,4.61);
\draw (-2.33,0.11)-- (-0.33,2.01);
\draw (-4.33,2.11)-- (-0.33,2.01);
\draw (-4.33,2.11)-- (-1.83,4.61);
\draw (-4.33,2.11)-- (1.17,4.61);
\draw (-3.5,-1.1) node[anchor=north west] {a) $|V|=9$, $m=4$};
\draw (7.5,-1.1) node[anchor=north west] {b) $|V|=10$, $m=3$};
\draw (19.5,-1.1) node[anchor=north west] {c) $|V|=11$, $m=2$};
\draw (-5,-12.4) node[anchor=north west] {d) $|V|=10$, $m=3$};
\draw (8,-12.4) node[anchor=north west] {e) $|V|=8$, $m=2$};
\draw (20,-12.4) node[anchor=north west] {f) $|V|=8$, $m=1$};
\begin{scriptsize}
\fill [color=black] (11.62,3.49) circle (2.0pt);
\fill [color=black] (11.62,1.49) circle (2.0pt);
\fill [color=black] (13.52,0.87) circle (2.0pt);
\fill [color=black] (14.7,2.49) circle (2.0pt);
\fill [color=black] (13.52,4.11) circle (2.0pt);
\fill [color=black] (10.05,4.74) circle (2.0pt);
\fill [color=black] (8.1,4.29) circle (2.0pt);
\fill [color=black] (7.24,2.49) circle (2.0pt);
\fill [color=black] (8.1,0.69) circle (2.0pt);
\fill [color=black] (10.05,0.24) circle (2.0pt);
\fill [color=black] (0.68,-9.56) circle (2.0pt);
\fill [color=black] (0.68,-7.56) circle (2.0pt);
\fill [color=black] (-0.61,-6.02) circle (2.0pt);
\fill [color=black] (-2.58,-5.68) circle (2.0pt);
\fill [color=black] (-4.31,-6.68) circle (2.0pt);
\fill [color=black] (-4.99,-8.56) circle (2.0pt);
\fill [color=black] (-4.31,-10.44) circle (2.0pt);
\fill [color=black] (-2.58,-11.44) circle (2.0pt);
\fill [color=black] (-0.61,-11.09) circle (2.0pt);
\fill [color=black] (3.68,-8.41) circle (2.0pt);
\fill [color=black] (20.12,0.61) circle (2.0pt);
\fill [color=black] (22.12,0.61) circle (2.0pt);
\fill [color=black] (23.12,2.34) circle (2.0pt);
\fill [color=black] (22.12,4.08) circle (2.0pt);
\fill [color=black] (20.12,4.08) circle (2.0pt);
\fill [color=black] (19.12,2.34) circle (2.0pt);
\fill [color=black] (24.12,0.61) circle (2.0pt);
\fill [color=black] (26.12,0.61) circle (2.0pt);
\fill [color=black] (27.12,2.34) circle (2.0pt);
\fill [color=black] (26.12,4.08) circle (2.0pt);
\fill [color=black] (24.12,4.08) circle (2.0pt);
\fill [color=black] (12.71,-10.83) circle (2.0pt);
\fill [color=black] (13.87,-8.37) circle (2.0pt);
\fill [color=black] (12.67,-5.93) circle (2.0pt);
\fill [color=black] (10.01,-5.35) circle (2.0pt);
\fill [color=black] (7.9,-7.06) circle (2.0pt);
\fill [color=black] (7.92,-9.78) circle (2.0pt);
\fill [color=black] (10.06,-11.46) circle (2.0pt);
\fill [color=black] (15.87,-8.37) circle (2.0pt);
\fill [color=black] (18.91,-8.43) circle (2.0pt);
\fill [color=black] (20.91,-10.43) circle (2.0pt);
\fill [color=black] (22.91,-8.43) circle (2.0pt);
\fill [color=black] (20.91,-6.43) circle (2.0pt);
\fill [color=black] (23.91,-8.43) circle (2.0pt);
\fill [color=black] (27.91,-8.43) circle (2.0pt);
\fill [color=black] (25.91,-10.43) circle (2.0pt);
\fill [color=black] (25.91,-6.43) circle (2.0pt);
\fill [color=black] (1.17,4.61) circle (2.0pt);
\fill [color=black] (-1.83,4.61) circle (2.0pt);
\fill [color=black] (-0.33,2.01) circle (2.0pt);
\fill [color=black] (3.67,3.61) circle (2.0pt);
\fill [color=black] (1.67,0.11) circle (2.0pt);
\fill [color=black] (3.67,2.11) circle (2.0pt);
\fill [color=black] (-4.33,3.61) circle (2.0pt);
\fill [color=black] (-2.33,0.11) circle (2.0pt);
\fill [color=black] (-4.33,2.11) circle (2.0pt);
\end{scriptsize}
\end{tikzpicture}

\caption{Examples of elements of $\Hm$.}
\label{fig:hyex}

\end{figure}
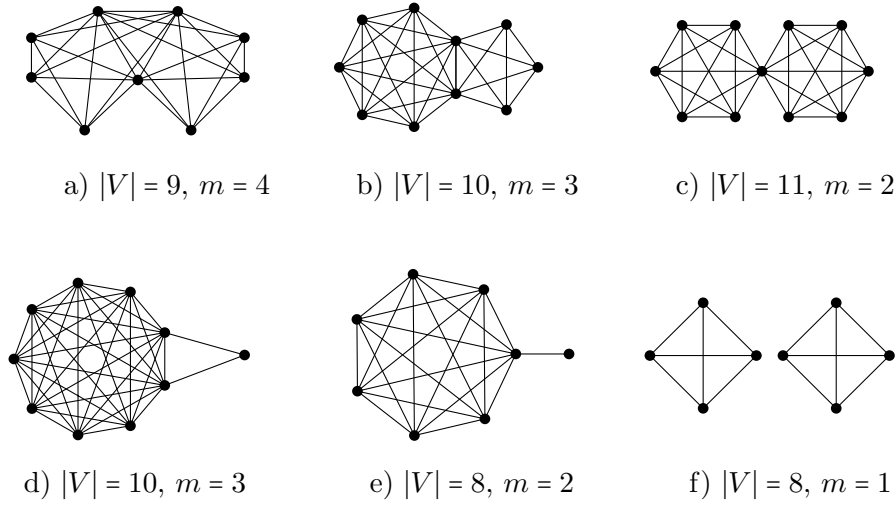





\begin{Question}\label{qu:theone}
  Is every matroid $\A$ on the ground set $K(V)$ which satisfies
  $\Hm\subseteq \Hy(\A)$ an $m$-dimensional abstract rigidity matroid?
\end{Question}

Evidence towards an affirmative answer to this question is given in
the next two propositions.

 \begin{Proposition} \label{Hsubset} \ \\
 Let $\Am$ be a $m$-dimensional abstract rigidity matroid on $K(V)$ with hyperplanes $\Hy$.
 Then $\Hm \subseteq \Hy$.
\end{Proposition}
\begin{proof}
 Let $H\in\Hm$ such that $H=K(V_1)\cup K(V_2)$ with $V_1,V_2\subseteq V$, $V_1\cup V_2 = V$, $|V_1\cap V_2|=m-1$ and
 $V_i\nsubseteq V_j$ for $i\neq j$. We have:
 \begin{align*}
  \cl{H}=\cl{K(V_1) \cup K(V_2)} \stackrel{(C1)}{\subseteq} K(V(K(V_1)))\cup K(V(K(V_2))) \\
  = K(V_1) \cup K(V_2) = H \subseteq \cl{H}.
 \end{align*}
 Thus $\cl{H}=H\neq K(V)$.\\
  Let $V':= V_1\cap V_2 $ and let $e\in K(V)\setminus H$, $e=v_1v_2$.\\
  Then w.l.o.g. $v_i\in V_i$ ($i=1,2$), $v_1,v_2\notin V'$.\\
  Note that $K(V'\cup \{v_1, v_2\})$ is as a complete edge-set rigid, as is $K(V_1)$.
  In addition $|V(K(V'\cup \{v_1,v_2\}))\cap V_1|=|V'\cup \{v_1\}|=m$, so by (C2), 
  also their union $K(V'\cup \{v_1,v_2\})\cup K(V_1)$ is rigid, as is $K(V'\cup \{v_1,v_2\})\cup K(V_2)$.
  The union of these sets is $H\cup e$ and
\[ \cl{H\cup\{ e \}}= \cl{(K(V'\cup \{ v_1,v_2\})\cup K(V_1)) \cup K((V'\cup \{v_1,v_2\})\cup K(V_2))} \stackrel{(C2)}{=} K(V). \]
\end{proof}

\begin{remark}
 Notice that, in general, the inclusion $\Hm\subseteq \Hy(\A)$ is strict. The figure below illustrates this.
  \begin{figure}[h]
\centering
\definecolor{qqzzzz}{rgb}{0,0.0,0}
\begin{tikzpicture}[line cap=round,line join=round,>=triangle 45,x=0.28cm,y=0.28cm]
\clip(-4,-2.7) rectangle (36,7.5);
\draw (0,0)-- (0,6);
\draw (0,6)-- (-3,3);
\draw (-3,3)-- (0,0);
\draw (0,0)-- (4,0);
\draw (4,0)-- (7,3);
\draw (7,3)-- (4,6);
\draw (4,6)-- (0,0);
\draw (4,6)-- (4,0);
\draw (0,0)-- (7,3);
\draw (14,0)-- (11,3);
\draw (11,3)-- (14,6);
\draw (14,0)-- (14,6);
\draw (18,0)-- (14,6);
\draw (18,0)-- (14,0);
\draw (18,0)-- (11,3);
\draw (18,0)-- (18,6);
\draw (21,3)-- (18,6);
\draw (21,3)-- (18,0);
\draw (28,0)-- (28,6);
\draw (28,6)-- (25,3);
\draw (28,0)-- (25,3);
\draw (32,6)-- (32,0);
\draw (32,0)-- (35,3);
\draw (35,3)-- (32,6);
\draw [line width=0.4pt] (32,0)-- (28,0);
\draw [line width=1.2pt,color=qqzzzz] (32,6)-- (28,6);
\draw [line width=1.2pt,dash pattern=on 2pt off 6pt,color=qqzzzz] (18,6)-- (14,6);
\draw [line width=1.2pt,dash pattern=on 2pt off 6pt,color=qqzzzz] (4,6)-- (0,6);
\begin{scriptsize}
\fill [color=black] (0,0) circle (2pt);
\draw[color=black] (0,-.85) node {5};
\fill [color=black] (-3,3) circle (2pt);
\draw[color=black] (-3.5,3.75) node {6};
\fill [color=black] (0,6) circle (2pt);
\draw[color=black] (0,6.94) node {1};
\fill [color=black] (4,6) circle (2pt);
\draw[color=black] (4,6.94) node {2};
\fill [color=black] (4,0) circle (2pt);
\draw[color=black] (4,-.85) node {4};
\fill [color=black] (7,3) circle (2pt);
\draw[color=black] (7.32,3.75) node {3};
\fill [color=black] (11,3) circle (2pt);
\draw[color=black] (10.5,3.75) node {6};
\fill [color=black] (14,0) circle (2pt);
\draw[color=black] (14,-.85) node {5};
\fill [color=black] (14,6) circle (2pt);
\draw[color=black] (14,6.94) node {1};
\fill [color=black] (18,6) circle (2pt);
\draw[color=black] (18,6.94) node {2};
\fill [color=black] (18,0) circle (2pt);
\draw[color=black] (18,-.85) node {4};
\fill [color=black] (21,3) circle (2pt);
\draw[color=black] (21.31,3.75) node {3};
\fill [color=black] (25,3) circle (2pt);
\draw[color=black] (24.5,3.75) node {6};
\fill [color=black] (28,0) circle (2pt);
\draw[color=black] (28,-.85) node {5};
\fill [color=black] (28,6) circle (2pt);
\draw[color=black] (28,6.94) node {1};
\fill [color=black] (32,6) circle (2pt);
\draw[color=black] (32,6.94) node {2};
\fill [color=black] (35,3) circle (2pt);
\draw[color=black] (35.33,3.75) node {3};
\fill [color=black] (32,0) circle (2pt);
\draw[color=black] (32,-.85) node {4};
\draw[color=qqzzzz] (30.11,6.94) node {$x$};
\draw[color=qqzzzz] (16.09,6.94) node {$x$};
\draw[color=qqzzzz] (2.1,6.94) node {$x$};
\draw[color=black] (2.1,-2.2) node {$H_1$};
\draw[color=black] (16.09,-2.2) node {$H_2$};
\draw[color=black] (30.11,-2.2) node {($H_1\cap H_2)\cup x$};
\end{scriptsize}
\end{tikzpicture}

\caption{An example where $|V|=6$ and $m=2$.}

\end{figure}
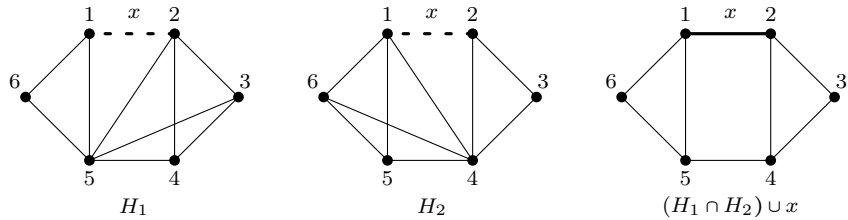

  The edge-sets on the left and in the middle are elements of $\mathcal{H}_2(V)$ and $x$ is not in their union. 
  By the hyperplane axioms for matroids, the set on the right has to be subset of another hyperplane which is at least 
  $2$-vertex-connected, hence not an element of $\mathcal{H}_2(V)$. 
\end{remark}



 \begin{Proposition} \label{HC1} \ \\
 Let $\A$ be matroid on $K(V)$ with hyperplanes $\Hy$ and $\Hm \subseteq \Hy$.
 Then its closure-operator fulfills (C1).
\end{Proposition}
\begin{proof}
 Let $E,F\subseteq K(V)$ such that $|V(E)\cap V(F)|<m$. 
 For given $v\in V$, the complement of $\vstar(v)$ minus $m-1$ arbitrary edges is a hyperplane. 
 So we are able to find hyperplanes which intersect into $K(V\setminus\{v\})$.
 This means that for all $v\in V$, the set $K(V\setminus\{v\})$ is closed in $\A$ and thus is every complete edgeset.
 Hence $\cl{E}\subseteq K(V(E))$ and $\cl{F}\subseteq K(V(F))$, so we may assume w.l.o.g. that 
 $V(E)\nsubseteq V(F)$ and vice versa. 
 Now we are able to choose $V_1,V_2\subseteq V$ such that $K(V_1)\cup K(V_2)$ is an element of $\Hm$ - a hyperplane -
 and 
 \[ \left( K(V_1)\cup K(V_2) \right) \cap \left( K(V(E\cup F)) \right) = K(V(E))\cup K(V(F)). \]
 Because the sets which intersect lefthandside are closed in $\A$, $ K(V(E))\cup K(V(F))$ must also be closed. 
 It contains $E\cup F$, so $\cl{E\cup F}\subseteq K(V(E))\cup K(V(F))$.
\end{proof}

We conclude this section with one more token of the advantage of
considering the set $\Hm$: the following proposition gives an
immediate proof, for all dimension, of a strenghtening of the problem
posed in \cite[Exercise 4.7.]{combrig} for dimension $2$.

\begin{Proposition} \label{connect} \ \\
 Let $\A$ be a matroid on $K(V)$, with hyperplanes $\Hy$ and $\Hm\subseteq\Hy$.
 Furthermore let $E$ be a rigid edgeset (i.e. $\cl{E}=K(V(E))$). Then $(V(E),E)$ is at least $m$-vertex-connected.
\end{Proposition}
\begin{proof}
Otherwise choose a vertex-set $U\subseteq V(E)$ such that $|U|<m$ and deleting $U$ out of $(V(E),E)$ would leave a disconnected graph.
Like in the proof before, we are able to find $V_1,V_2\subseteq V$ such that $K(V_1)\cup K(V_2)$ is a hyperplane in $\Hm$, 
$E\subseteq K(V_1)\cup K(V_2)$ and $U\subseteq V_1\cap V_2$.
We can construct $V_1,V_2$ in such a way, that at least one vertex $v_1$ of one component of the graph after deleting $U$ lies in $V_1\setminus V_2$, 
while another vertex $v_2$ of another such component lies in $V_2\setminus V_1$.
This implies that $v_1,v_2\in V(E)$ but $v_1v_2\notin K(V_1)\cup K(V_2)$. Therefore $v_1v_2\notin \cl{E}=K(V(E))$. This is a contradiction.
\end{proof}

\subsection{An inductive characterization}

The last part of our study moves form the observation that certain
restrictions of rigidity matroids are again rigidity matroids.

\begin{Definition}
  Let $\A$ be an $m$-dimensional abstract rigidity matroid on the set
  $K(V)$. Given $X\subseteq V$, let $\A[X]$ denote the restriction of
  $\A$ to the set $K(X)$.
\end{Definition}

\begin{Lemma}\label{Lem:her}
  Let $\A$ be an $m$-dimensional abstract rigidity matroid on $V$. For
  every $X\subseteq V$, $\vert X \vert \geq m+1$, $\A[X]$ is an
  $m$-dimensional rigidity matroid.
\end{Lemma}
\begin{proof}
  We check condition (ii) and (iii) of Proposition
  \ref{Prop6}. Condition (ii)
  is clearly inherited from $\A$ (and possibly empty) and Lemma
  \ref{lem:rank} implies (iii).
\end{proof}

The following definition identifies the ``prescribed substructures''
through which we will be able to give an intrinsic characterization of
abstract rigidity matroids.

\begin{Definition}
  Recall Definition \ref{df:Hm} and let
$$
\He:=\{  K(V_1) \cup K(V_2) \in \Hm \mid \vert V_1\vert = m
\}=\{\Delta_v^A \mid \vert A \vert = m-1\}
$$
\end{Definition}

The fact that the subfamily $\He$ enforces nontrivial structure on any
matroid with $\He\subseteq \Hy(\A)$can be
already seen from Lemma \ref{extensionlemma}, whose premise can be
rephrased as ``$\He\subseteq \Hy(\A)$''. Moreover, the same
requirement is equivalent to condition (i) in Proposition
\ref{Prop6} - hence, answering our question requires a detailed study
of the rank function of matroids on
$K(V)$ for which $\Hm^{(1)}$ is a subset of the hyperplane set. 

Our main result in this section is the following.

\begin{Theorem}\label{theo:2dim}
  Let $\vert V \vert \geq m+1$ and $\A$ a matroid on $K(V)$.
 The matroid $\A$ is an $m$-dimensional abstract rigidity matroid if
 and only if, for all $X\subseteq V$, $$\Hm^{(1)} (K(X))\subseteq
  \Hy(\A[X]).$$
\end{Theorem}

\begin{remark}
  The proof will actually work also with milder assumptions: it is
  enough to assume that there exists an enumeration $v_1,\ldots,v_k$
  of the vertices such that $\Hy_2^{(1)} (K(\{v_1,\ldots
  ,v_j\}))\subseteq \Hy(\A[\{v_1,\ldots ,v_j\}])$ for all $j=2,\ldots
  k$.
\end{remark}

The proof of the theorem relies on the following two general lemmas.

\begin{Lemma}\label{lem:twoparts}
    Let $m\geq 2$, $\vert V \vert \geq m+1$ and let $\A$ be a matroid on the
  ground set $K(V)$. Choose any $v_0\in V$.
  \begin{itemize}
  \item If all sets in $\He$ are closed in $\A$, $K(V\setminus v_0)$
    is closed in $\A$.
  \item If $\He\subseteq \Hy(\A)$, the rank of the flat $K(V\setminus v_0)$ in $\A$ is $r(K(V))-m$.
  \end{itemize}
\end{Lemma}

\begin{proof}
  Choose $v_0\in V$ and let $A:=V\setminus v_0$. Choose $a_1,\ldots,
  a_{m}\in A$ and let $A_0:=
  \{a_1,\ldots,a_{m)}\}$. Moreover, set $A_j:=\{v_0\}\cup A_0\setminus
  a_j$.

For $j=1,\ldots m-1$ let 
$$H_j:=
\Delta_{v_0}^{A_j}\in \He.$$
and 
$$
F_j:=\bigcap_{i\leq j} H_i.
$$
Moreover, set $F_0:=K(V)$ and let $r_0=r_\A(F_0)$. 

As intersections of hyperplanes, all $F_j$ are flats, so we have a
chain
$$
F_m\subsetneq F_{m-1} \subsetneq \ldots \subsetneq F_0
$$
in the lattice of flats of $\A$. Since  $K(V\setminus v_0)=F_{m}$, we
have proved the forst claim. 

For the second it will be enough to prove that $r_\A(F_j) = r_0 - j$.
First notice that, for
every $j>1$, $F_{j}$ covers $F_{j-1}$. In fact, by definition $F_0$
covers $F_1$ and, for $j>0$,
$$
F_j=K(A) \cup K(\{v_0, a_{j+1},\ldots, a_{m}\})
$$
so we see that the cardinality $\vert F_j\setminus F_{j-1}\vert
=1$. Therefore clearly there can be no $G$ with $F_j\subsetneq G \subsetneq
F_{j+1}$. The chain $F_m\subsetneq \ldots \subsetneq F_0$ is therefore
saturated, and its length thus equals the difference between the rank
of $F_0$ and the rank of $F_m$.
\end{proof}

\begin{Lemma}\label{lem:bottom}
    Let $m\geq 2$, $\vert V \vert \geq m+1$ and let $\A$ be a matroid on the
  ground set $K(V)$. If all sets in $\He$ are closed in $\A$, $K(V')$
    is independent in $\A$ for every $V'\subseteq V$ with $\vert V'
    \vert  = m$.
\end{Lemma}
\begin{proof}
Suppose by way of contradiction that $K(V')$ contains a circuit $C$, let
$uv$ be an element of this circuit and $A:=(V' \setminus \{v,u\}) \cup v_0$,
where $v_0$ is any element of $V\setminus V'$ (nonempty for
cardinality reasons). 
Then, $\Delta^A_v$ is closed in $\A$ and contains all of $K(V')$
except $uv$, a contradiction to the fact that $uv$ is in a circuit of $K(V')$.
\end{proof}

\begin{proof}[Proof of Theorem \ref{theo:2dim}]
One direction (``left to right'') is proved by Lemma
\ref{Lem:her}. For the other direction let $V_1\subseteq V_2\subseteq
\ldots \subseteq V_n$ be a filtration of $V$ such that $\vert V_i\vert = i$ for all
$i$. notice that the hereditarity
assumption allows us to apply Lemma \ref{lem:twoparts} recursively to get that 
$$
r(K(V_i))= r(K(V)) - m(\vert V \vert - \vert V_i \vert) 
$$


for all $i\geq m$, and with Lemma \ref{lem:bottom} we have
$r(K(V_m)={m \choose 2}$. Thus we can write
$$
{m \choose 2}= r(K(V)) - m(\vert V \vert - m ) 
$$
hence
$$
r(K(V)) = m\vert V \vert - m^2 + (m^2 -m)/2 =
m\vert V \vert - { m+1 \choose 2} 
$$
  and we conclude with Proposition \ref{Prop6}.
\end{proof}

\begin{Corollary}
  Question \ref{qu:theone} is now equivalent to the following 
  \begin{Question}
    Let $m\in \mathbb N_{>0}$ and $V$ a set with $\vert V \vert \geq
    m+1$. Is it true that, for any matroid $\A$ on the vertex set $K(V)$,
    $\Hm\subseteq \Hy(\A)$ implies $\He\subseteq \Hy(\A[X])$ for all
    $X$ with $\vert X\vert \geq m+1$? 
  \end{Question}
\end{Corollary}

\bibliographystyle{plain}
\bibliography{rigidity}{}

\begin{thebibliography}{1}

\bibitem{CraRo}
Henry~H. Crapo and Gian-Carlo Rota.
\newblock {\em On the foundations of combinatorial theory: {C}ombinatorial
  geometries}.
\newblock The M.I.T. Press, Cambridge, Mass.-London, preliminary edition, 1970.

\bibitem{combrig}
Jack Graver, Brigitte Servatius, and Herman Servatius.
\newblock {\em Combinatorial rigidity}, volume~2 of {\em Graduate Studies in
  Mathematics}.
\newblock American Mathematical Society, Providence, RI, 1993.

\bibitem{graver}
Jack~E. Graver.
\newblock Rigidity matroids.
\newblock {\em SIAM J. Discrete Math.}, 4(3):355--368, 1991.

\bibitem{km2co}
Jack~E. Graver, Brigitte Servatius, and Herman Servatius.
\newblock Abstract rigidity in {$m$}-space.
\newblock In {\em Jerusalem combinatorics '93}, volume 178 of {\em Contemp.
  Math.}, pages 145--151. Amer. Math. Soc., Providence, RI, 1994.

\bibitem{laman}
G.~Laman.
\newblock On graphs and rigidity of plane skeletal structures.
\newblock {\em J. Engrg. Math.}, 4:331--340, 1970.

\bibitem{nguyen}
Viet-Hang Nguyen.
\newblock On abstract rigidity matroids.
\newblock {\em SIAM J. Discrete Math.}, 24(2):363--369, 2010.

\bibitem{oxley}
James Oxley.
\newblock {\em Matroid theory}, volume~21 of {\em Oxford Graduate Texts in
  Mathematics}.
\newblock Oxford University Press, Oxford, second edition, 2011.

\bibitem{prism}
Sachin Patkar, Brigitte Servatius, and K.~V. Subrahmanyam.
\newblock Abstract and generic rigidity in the plane.
\newblock {\em J. Combin. Theory Ser. B}, 62(1):107--113, 1994.

\end{thebibliography}

\end{document}